\documentclass[12pt,a4paper]{article}

\usepackage{latexsym}
\usepackage{amssymb}
\usepackage{amsmath}
\usepackage{amsthm}
\usepackage{graphicx}
\usepackage[nohug,balance,midshaft]{diagrams}
\usepackage{url}

\numberwithin{equation}{section}

\usepackage[scale=0.70]{geometry}

\usepackage{sectsty}
\sectionfont{\center \sc \large}
\subsectionfont{\bf \large}
\subsubsectionfont{\it \large}

\usepackage{setspace}

\usepackage{mathbbol}

\newcommand{\singlespaced}{\renewcommand{\baselinestretch}{1}\normalfont}

\newtheorem{theorem}{Theorem}[section]
\newtheorem{corollary}[theorem]{Corollary}
\newtheorem{lemma}[theorem]{Lemma}
\newtheorem{proposition}[theorem]{Proposition}

\theoremstyle{definition}
\newtheorem{definition}[theorem]{Definition}

\theoremstyle{remark}
\newtheorem{remark}[theorem]{Remark}

\numberwithin{equation}{section}

\newcommand{\fnamesize}{\footnotesize}
\newcommand{\fname}[1]{\mbox{\fnamesize #1}}

\newcommand{\mc}{\mathcal}

\newcommand{\Pic}{\textrm{Pic}}

\newcommand{\rk}[1]{\mbox{rk}(#1)}

\newcommand{\isom}{\cong}
\newcommand{\intersect}{\cap}

\newcommand{\poincare}{\mathcal{P}}
\newcommand{\KQ}{\mathcal{Q}}

\newcommand{\slope}{\mu}

\newcommand{\OO}{\mathcal{O}}

\newcommand{\tensor}{\otimes}
\newcommand{\dsum}{\oplus}
\newcommand{\Dsum}{\bigoplus}

\newcommand{\union}{\bigcup}
\newcommand{\infinity}{\infty}
\newcommand{\id}{id}

\newcommand{\proj}[1]{\mathbb{P}^#1}

\newcommand{\Heis}{\mathbb{H}_8}
\newcommand{\Z}{\mathbb{Z}}

\newcommand{\C}{\mathbb{C}}
\newcommand{\Q}{\mathbb{Q}}

\newcommand{\dual}[1]{\widehat{#1}}
\newcommand{\sdual}[1]{\hat{#1}}

\newcommand{\dV}{\dual{V}}
\newcommand{\dH}{\dual{H}}
\newcommand{\dA}{\dual{A}}
\newcommand{\dc}{\sdual{c}}
\newcommand{\dl}{\sdual{l}}
\newcommand{\de}{\sdual{e}}
\newcommand{\dE}{\sdual{E}}

\newcommand{\rest}{\arrowvert}
\newcommand{\res}[1]{\widetilde{#1}}

\newcommand{\FM}{\text{S}}
\newcommand{\fm}{\text{s}}
\newcommand{\fund}[1]{\pi_1(#1)}

\newcommand{\ses}[3]{0 \longrightarrow #1 \longrightarrow #2 \longrightarrow #3 \longrightarrow 0}

\newcommand{\euler}{\chi}

\begin{document}

\singlespaced
\pagenumbering{arabic}

\begin{titlepage}

{\noindent \Large \bf \fontfamily{pag}\selectfont The Spectral Construction for a (1,8)-Polarized Family of Abelian Varieties}

\vspace*{0.5cm}
\noindent \rule{\linewidth}{0.5mm}

\vspace*{1.3cm}

{\noindent  \fontfamily{pag}\selectfont  Anthony Bak}\\[0.2em]
{\small {\it \indent Department of Mathematics and Statistics\\
\indent Mount Holyoke College\\
\indent South Hadley, MA 01075, USA}\\[0.1em]
\indent \url{abak@mtholyoke.edu}}\\[0.3em]

\vspace*{1.5cm}

\hspace*{1cm}
\parbox{11.5cm}{{\sc Abstract:} We extend the Spectral Construction, a technique used with great success to study and construct vector bundles on elliptically fibered varieties, to a special family of abelian surface fibered varieties. The results are motivated by requirements from Heterotic String Phenomenology where vector bundles with specified chern class are required to produce a realistic particle spectrum. Although only certain (1,8)-polarized families of abelian surfaces are considered here we expect the main ideas to carry over to other families of abelian varieties with a uniform relative polarization type.}

\end{titlepage}

\setstretch{1.08}
\tableofcontents 

\section{Introduction}

Two decades ago it was shown that a physically realistic string theory can be realized with a Calabi-Yau threefold as the compact portion of space-time \cite{GreenSchwarzWitten}. On this Calabi-Yau a Heterotic String theory compactification is specified by giving a vector bundle with special properties determined by physics. In the last decade, many examples of Calabi-Yau manifolds have been produced, but until recently \cite{BouchardDonagi} there were no examples of a Heterotic String with a phenomenologically correct particle spectrum.  

This paper examines a particular family of Calabi-Yaus and develops the theory of Fourier-Mukai transforms and the Spectral Construction to build vector bundles with specified chern characters.  The Calabi-Yaus we examine are of intrinsic mathematical interest because of some extremizing properties, and also form promising candidates for physical reasons.

In the companion paper \cite{BakBouchardDonagi} we explore the physical consequences of using the constructions described here.  Unfortunately the main result is that under some mild restrictions we have a ``no-go'' result stating the the weak heterotic challenge can not be satisfied (see section \ref{PhysicalMotivations}).

\subsection{Main Example}

The family of Calabi-Yaus considered here arise from explicit constructions of the moduli of $(1,8)$-polarized abelian varieties with level structure, originally considered in some detail in \cite{GrossPopescu}. Each Calabi-Yau in the family is an abelian surface fibration $V \rightarrow \proj{1}$ with a relative (1,8) polarization class, H. We will use quotients of $V$ by a freely acting group $G$ as our target variety $\dual{V}$ on which we examine the moduli of vector bundles. In both cases the fibration is flat but not smooth and the types of singular fibers we encounter are generic in boundary components of compactification of the moduli of $(1,8)$-polarized abelian surfaces.

Both $V$ and $\dual{V}$ have a number of appealing characteristics that make them interesting to both Mathematicians and Physicists. The variety $\dual{V}$ has a fundamental group of order $64$. Most of the known Calabi-Yau varieties are simply connected and those considered here have some of the largest known fundamental groups. Recently $V$ was shown \cite{GrossPavanelli} to have the largest known Brauer Group of a Calabi-Yau threefold and it has been conjectured that $V$ is the universal cover of its mirror \cite{Pavanelli}. 

For physics, large fundamental groups give more flexibility when breaking the grand unified group to the standard model group. This could represent a significant improvement on previous attempts \cite{BouchardDonagi}. There are families of Calabi-Yaus closely related to those considered here to which many of the methods used in this paper apply. Some of those examples are known to have non abelian fundamental groups and are discussed briefly in concluding remarks.

\subsection{Physical Motivations: The Weak Heterotic Challenge}\label{PhysicalMotivations}

The exploration of these families of Calabi-Yaus is motivated by problems in the compactification of a Heterotic String theory to four dimensions where bundles with particular characteristics are required. For an $E_8 \times E_8$ Heterotic theory compactified to a four dimensional theory we need to satisfy the weak heterotic challenge \cite{DonagiOvrutPantevWaldram,BakBouchardDonagi}:
\begin{enumerate}
	\item[$\bullet$] a Calabi-Yau threefold X with K\"{a}hler form $\omega$
	\item[$\bullet$] An $E_8$ bundle $\mathcal{E} \longrightarrow X$ whose structure group reduces to $G$
	\item[$\bullet$] The centralizer of $G$ in $E_8$ contains the standard model group $SU(3)\times SU(2) \times U(1)$ as a direct summand. 
\end{enumerate}
\noindent In practice we will have $G = SU(4)$ or $SU(5)$ and we'll replace the principal bundle $\mathcal{E}$ by the associated vector bundle which we'll continue to refer to as $\mathcal{E}$.  We have:
\begin{enumerate}
  	\item[$\bullet$] $rk(\mathcal{E}) = 4$ or $5$.
	\item[$\bullet$]  $\mathcal{E}$ should be Mumford polystable with respect to $\omega$ (Supersymmetry Preservation)
	\item[$\bullet$] $c_1(\mathcal{E}) = 0$
	\item[$\bullet$]  $c_2(T_X) - c_2(\mathcal{E}) =  \text{\;\{class of effective curve\}}$ (Anomaly Cancellation)
	\item[$\bullet$] $c_3(\mathcal{E}) = 6$ (3 Generations Condition) 
\end{enumerate}
These numerical requirements are quite rigid, and the search for bundles satisfying the constraints typically requires detailed analysis of the particular threefold in question using the techniques described above.

\subsection{Outline of the Spectral Construction for Vector Bundles}

Using the above constraints we can calculate possible Chern characters of a desired bundle $\mathcal{E}$. Then using our explicit calculation of the action of the Fourier-Mukai transform on cohomology we can calculate the chern characters of a sheaf $\mathcal{N}$ such that the Fourier-Mukai of it gives a sheaf with the desired chern classes.

In some cases $\mathcal{N}$ can be written as a line bundle supported on a smooth curve finite over the base.  In those cases $FM(\mathcal{N})$ will be a vector bundle.  Some properties of the vector bundle are related to geometric properties of a the curve, in particular we are able to get conditions on the curve such the the resulting bundle is stable.

This is the general outline of how we would like to proceed with the Spectral Construction.  In the example consider here we are unable to make this construction or variations on it produce a bundle meeting the weak heterotic challenge and under mild hypothesis have shown that a vector bundle with the desired properties does not exist \cite{BakBouchardDonagi}.

\subsection*{Acknowledgments}

I would like to thank Ron Donagi for his patience while advising me on the project.  I would also like to thank Tony Pantev and Mark Gross for their many useful suggestions, discussions and comments without whose help this project would not have started.  This work was supported by the University of Pennsylvania and The Max Planck Institute for Mathematics.

\section{The Geometry of $V_{8,y}$ and $\dual{V}$}
We review the known results about $V_{8,y}$ and give some new results for $\dual{V}$.  The goal of this chapter is to provide a basic geometric understanding of our main example by calculating the relevant cohomology groups and intersection products.  First we'll summarize the relevant details for our Calabi-Yau threefold $V=V^1_{8,y}$. We closely follow the notation and exposition of \cite{GrossPavanelli}.  More details can be found in \cite{GrossPopescu} and \cite{GrossPavanelli}.

\subsection{A Relative (1,8)-Polarized Family of Abelian Surfaces:  $V_{8,y}$ }

Write $V_{8,y}$ for the  $(2,2,2,2)$ complete intersection in $\proj{7} = [x_0:\ldots:x_7]$ depending on a parameter $y = [y_1,y_2,y_3] \in \proj{2}_- \hookrightarrow \proj{7}$ given by
\begin{equation}
 \begin{array}{l}
y_1y_3(x_0^2+x_4^2) - y_2^2(x_1x_7+x_3x_5)+(y_1^2+y_3^2)x_2x_6=0,\\
y_1y_3(x_1^2+x_5^2) - y_2^2(x_2x_0+x_4x_6)+(y_1^2+y_3^2)x_3x_7=0,\\
y_1y_3(x_2^2+x_6^2) - y_2^2(x_3x_1+x_5x_7)+(y_1^2+y_3^2)x_4x_0=0,\\
y_1y_3(x_3^2+x_7^2) - y_2^2(x_4x_2+x_6x_0)+(y_1^2+y_3^2)x_5x_1=0.\\
\end{array}
\end{equation}

There is an action of the Heisenberg group $\Heis$ generated by the elements $\sigma$, $\tau$ given via an action on the homogeneous coordinates given by $\sigma(x_i) = x_{i-1}$ and $\tau(x_i) = \xi^{-i}x_i$ where $\xi$ is a primitive $8^{th}$ root of unity and the indices on the coordinates should be read mod 8.

The main results summarizing our description of $V$ are the following ( \cite{GrossPavanelli} theorem 1.1 and Proposition 1.3)

\begin{theorem} Let $y \in \proj{2}_-$ be general.
\begin{enumerate}
\item $V_{8,y}$ is a complete intersection singular at 64 ordinary double points.  These 64 points are the $\Heis$ orbit of $y$.
\item There is a small resolution $V^1_{8,y} = V \longrightarrow V_{8,y}$ and a fibration $\pi : V \longrightarrow \proj{1}$ whose general fiber is an abelian surface A with a (1,8) polarization induced by its embedding in $\proj{7}$
\item $\chi (V) = 0$ and $h^{1,1}(V) = h^{1,2}(V) = 2$.
\item $H^2(V,\Z)/Tors$ is generated by the pullback of the hyperplane section, H and the class of the fiber of $\pi$, A.  $H^4(V,\Z)/Tors$ is generated be the class of an exceptional curve $e$ and the class of a line $l$ in $V_{8,y}$ disjoint from the singular locus.  Since $A.e=1$ any exceptional curve is a section of $\pi$.  The cohomology ring mod torsion is
\begin{center}
	$\begin{array}{cccc} H.e=0 & H.l=1 & A.l=0 & A.e=1 \end{array}$ \\
	$\begin{array}{cc} H^3=16 & H^2A=16 \end{array}$ \\
	$\begin{array}{ccc} H.A=16l & A^2=0 & H^2 = 16e + 16l \end{array}$
\end{center}
\end{enumerate}
\end{theorem}

\begin{proposition} For a general choice of $y$, $\pi$ has eight singular fibers, and they are all translation scrolls.
\end{proposition}

\subsection{The Dual Fibration: $\dual{V}$}

Given a smooth fiber $A$ and an ample (1,8) polarization $H_A$, we get a map $\phi_{H_A} : A \longrightarrow \dual{A}$, $x \mapsto t^\ast_xH \tensor H^{-1}$.  By standard theory this map is an isogeny with kernel isomorphic to $\Z_8 \times \Z_8$. Denote the kernel by $K_A(H)$.

\begin{proposition}\label{Dual} The action of the kernel extends to the singular fibers yielding a global quotient $\dual{V} = V/K(H)$ that agrees on the smooth fibers with the standard isogeny $A \rightarrow \dA$.
\end{proposition}
\begin{proof}
The action of $\Heis$ on $\proj{7}$ is not free but its fixed points are disjoint from the locus defining $V_{8,y}$  and hence acts freely there (see \cite{GrossPavanelli}). Further, the action is such that on smooth $A$ with ample line bundle $H_A$ it acts as translation by elements of the kernel $K(H)$.  By \cite{HulekKahnWeintraub} Propositions 5.40 and 5.41 this action extends to the singular fibers.

On a smooth fiber $K(H)$ induces an isogeny $A \longrightarrow \dA$ permuting the 64 sections $e_i$, and on the singular fibers $T_i$ it acts along the elliptic curve $E_i$ forming the singular locus  by translation by the points of order 8, this action moves each line $l$ in the singular fiber in an orbit of 64 other lines in the fiber (See comments before Theorem 1.4 in \cite{GrossPavanelli}).  Taking the quotient by this group action we get the dual fibration $\dual{\pi}:\dual{V} \longrightarrow B$.  Let $\phi_H: V \to \dual{V}$ denote the quotient map.
\end{proof}

\begin{proposition}\label{DualMult}
The hodge numbers of $V$ and $\dual{V}$ are the same and an integral basis for $H^\ast(\dual{V},\Z)$ is $\dH$, $\dA$, $\de$, $\dE$, $[pt]$, where $\dE$ is the class of a curve forming the singular locus of the singular fibers, and $8[\dE] = \dl$.  The relations between the classes on $\dual{V}$ and $V$ are.  
	\begin{displaymath}
		\begin{array}{cc} 
		 	\phi_H^\ast(\dH) = 8H 		& 	\phi_{H\ast}(H) = 8\dH \\
		 	\phi_H^\ast(\dA) = A 		& 	\phi_{H\ast}(A) = 64\dA \\
		 	\phi_H^\ast(\de) = 64e 		& 	\phi_{H\ast}(e) = \de \\
		 	\phi_H^\ast(\dl) = 64l 		& 	\phi_{H\ast}(l) = 8\dE \\
		 	\phi_H^\ast(\dE) = 8l 		&	\phi_{H\ast}([pt]) = [pt] \\
		 	\phi_H^\ast([pt]) =64[pt] 	& 				
		\end{array}
	\end{displaymath}
\end{proposition}

\begin{proof} Since $\pi_0(V)=0$, and $G=\Z_8 \times \Z_8$ acts freely on $V$ we can use  the Leray-Serre spectral sequence (\cite{Weibel} Theorem 6.10.10)
\begin{displaymath}
	E^{p,q}_2 = H^p(G,H^q(X,A)) \Longrightarrow H^{p+q}(X/G,A)
\end{displaymath}
to calculate the cohomology of $\dual{V}$. Using \cite{Weibel} Theorem 6.2.2 we can calculate that 
\begin{displaymath}
	H^n(\Z_8 \times \Z_8,\Q) = \left\{ \begin{array}{rr}
									\Q  &  \mbox{if $n=0$}     \\
									0   &  \mbox{otherwise}   \\
									\end{array} 
								\right.
\end{displaymath}
So the spectral sequence degenerates at $E^{p,q}_2$ with terms concentrated in the first column $E^{0,p} = H^0(\Z_8 \times \Z_8, H^p(X,\Q))$ and we conclude that the dimension of the cohomology mod torsion is as desired.

With the exception of the fact that $\phi_H^\ast(\dH) = 8H$ all pullback/pushforward results follow from the proof of \ref{Dual}. For $\phi_H^\ast \dH = 8H$ we use the following.  By general results for Abelian varieties, the dual to a variety with polarization type $(1,8)$ is also of the same type (\cite{BirkenhackeLange} Proposition 14.4.1).  So the smooth fibers $\dA$ of $\dual{V}$ are also of type $(1,8)$, write $\dH_{\dA}$ for the polarization divisor restricted to a smooth fiber.  Since $\dual{H}_{\dual{A}}$ is of type $(1,8)$ we have that $\dual{H}_{\dual{A}}.\dual{H}_{\dual{A}}=16$.   Write $aH=\phi_H^\ast \dH$ and consider the pullback $\phi_H^\ast (\dH_{\dA}^2) = \phi^\ast_H (\dH_{\dA}) .\phi^\ast_H (\dH_{\dA})=a^2 H.H.A = a^2 16$ on the other hand $\phi_H^\ast (\dH_{\dA}^2) = \phi^\ast_H (16[pt]) = 64  (16) [pt]$ from which we see that $a=8$
\end{proof}

\begin{corollary}

The intersection product in $H^\ast(\dual{V},\Z)$ mod torsion is
	\begin{center}
		$\begin{array}{ccccc}  \dH . \de = 0 &  \dH . \dl = 8 pt  &  \dH . \dE = 1 pt  &  \dA . \de =1 pt   &  \dA . \dl = 0 \end{array}$ \\
		$\begin{array}{ccc}  \dH . \dH = 16 \de  + 128 \dE    &  \dH . \dA = 16 \dE  &  \dA . \dA =0 \end{array}$ \\ 
		$\begin{array}{ccc}  \dH ^3= 128[pt]   &  \dH ^2. \dA  = 16 [pt]  &  \dA ^3=0 \end{array}$ \\
	\end{center}
\end{corollary}
\begin{proof} The proof is just a repeated application of the fact that for cohomology classes $x,y$ we have $\phi_H^\ast(x.y)=\phi_H^\ast(x).\phi_H^\ast(y)$ and the projection formula $\phi_{H\ast}(x.\phi_H^\ast y)=\phi_{H\ast}(x).y$ .  For example, $[\dH].[\dl] = [\dH].\phi_{H\ast}[l] = \phi_{H\ast}(\phi^\ast_H[\dH].[l]) = \phi_{H\ast}([8H].l]) = 8[pt]$. The others are similar. 
\end{proof}

\begin{proposition} The cone of effective curves on $\dual{V}$ is given by $a[\de] + b [\dl]$ for $a$,$b>0$
\end{proposition}
\begin{proof}  
For an effective curve $\dc$ write $[\dc] = -a[\de] + b[\dl]$ for its class.  Suppose $a > 0$ and we have $[\dc].[\dA] = -a$. Consider the map $H^0(\OO(A)) \rightarrow H^0(\OO(A)\rest_c)$.  We claim this map is not the zero map. To be the zero map we need that all global section of $\OO(A)$ vanish along the curve $\dc$.  That is to say that $\dc$ is in the fixed locus of the linear system $|A|$.  But $\OO(A) = \pi^\ast \OO_{\proj{1}}(pt)$ so the fibers are all in the same linear system and there is no fixed locus.  By contradiction the map $H^0(\OO(A)) \rightarrow H^0(\OO(A)\rest_c)$ is not the zero map.  Since it is not the zero map, we have that $H^0(\OO(A)\rest_c) \neq 0$  But $\OO(A)\rest_c$ is a line bundle of negative degree and thus has no global sections.  We conclude that $-a$ must be positive.  

For the case that $b<0$ we can make a similar argument once we note that $H$ gives the morphism $V\rightarrow \proj{7}$ and so cannot have a fixed locus.
\end{proof}

\begin{proposition}\label{Ample}For $k,l \geq 0$, we have $\dual{H}_0 = l\dH + k\dA$ is an ample divisor.
\end{proposition}
\begin{proof}
	By Kleimann's ampleness criteria we have that $\dH_0$ is a ample divisor iff
	\begin{eqnarray*}
		\dH_{0}^{>0} & = & \{C \in N_1(\dV) | C.\dH_0 > 0  \}		\\
				& \supset & {\overline{NE(\dV)}}\backslash \{0\} 
	\end{eqnarray*}
	Where $N_1(\dV)$ is the set of 1-cycles and $\overline{NE(\dV)}$ is the closure of the cone of effective curves.  This is clearly satisfied by our $\dH_0$.
\end{proof}

\begin{proposition}\label{EffInt}For $H_0$ an ample divisor given above, $D$ an effective divisor we have
\begin{equation*}
	D.H_0^i.A^{2-i} \geq 0
\end{equation*}
for $0 \leq i \leq 2$.
\end{proposition}
\begin{proof}
The case $i=0$ follows from the fact that $A^2=0$.  For $i=1$ note that as $H_0$ is ample, $D.H_0.A < 0$ implies that $D.A$ is not the class of an effective curve so $H^0(\OO_A(D\rest_A))=0$.  But if the map $H^0(\OO_V(D)) \rightarrow H^0(\OO_A(D\rest_A)$ given by the restiction map is zero for all A then $H^0(\OO_V(D)) = 0$.  The argument for $i=2$ is similar.
\end{proof}

Note that the same proof would have worked for the classes on $\dV$.  We end this section with a basic observation regarding the fundamental group of $\dV$.

\begin{proposition}
	$V$ is the universal cover of $\dual{V}$ and the fundamental group is given by $\fund{\dual{V}} = Z_8 \times Z_8$
\end{proposition}
\begin{proof}
	By the description of the action in \ref{Dual} its a free action so by standard algebraic topology $\fund{\dual{V}} = \Z_8 \times \Z_8$ 
\end{proof}

%

\section{Fourier-Mukai Transforms on Abelian Surface Fibrations}
The goal of this chapter is to calculate explicitly the descent of the Fourier-Mukai transform to a map on cohomology for our abelian surface fibration.  First we  examine the singular fibers to get an extension of the multiplication map on the smooth fibers.  We use the multiplication map to define a Poincare sheaf on the fiber product and hence a Fourier-Mukai transform. In the last section we calculate the Fourier-Mukai transform for a number of independent sheaves so that we can deduce the action on cohomology.  We begin with a review of the extension of the multiplication maps in the case of curves as it gives the essential idea.

\subsection{Review: Extension of Multiplication for the Nodal Cubic}
Before we proceed recall how this was done in the one dimensional case.  Take $N$ to be a nodal cubic curve.  By extension of multiplication mean a commutative diagram 

\begin{diagram}
	R				&	\rTo^{\res{m}}	&	N			\\
	\dTo			&			 		&	\dTo_{id}	\\
	N \times N 	& \rDashto^m 		& N 
\end{diagram}
where $m$ is defined as the regular multiplication map on the open set $\C^\ast \times \C^\ast$ formed by removing the singular point of N and R is birational to $N\times N$. 

\begin{proposition}For the nodal cubic N there is an extension of the multiplication map.
\end{proposition}

\begin{proof}
We start by resolving the nodal curves. We have the vertical maps in the diagram
\begin{diagram}
 	\proj{1} \times \proj{1} & \rDashto^{m'} & \proj{1} \\
	\dTo<{\fname{glue}}					&				 &	\dTo 			\\
	N \times N				& \rDashto^m & N
\end{diagram}

Explicitly we can write $m'$ as follows. Consider $\proj{1} \times \proj{1}$ embedded in $\proj{3}$ as a quadric $Q$ using the Segre embedding $[u,v],[s,t] \longmapsto [us,ut,vs,vt] = [x,y,z,w]$.  We will define $m'$ to be the projection from the line $l \isom \proj{1}$ given by the points $[0,l_1,l_2,0]$.  This will define a map outside of $Q \intersect l = \{[0,0,1,0],[0,1,0,0] \}$ to $L \isom \proj{1} = \{ [L_1,0,0,L_2] \} $.  The projection in the coordinates of $\proj{1} \times \proj{1}$ is $([u,v][s,t]) \mapsto [ua,vb]$. We note immediately that on the open set $\C^\ast \times \C^\ast \rightarrow \C^\ast$ this map is exactly the multiplication map on $\C^\ast$ where $\C^\ast \hookrightarrow \proj{1}$ as the set of points $[u,v]$ with both u and v not equal to zero. Write $0$ for the point $[0,1]$ and $\infinity$ for $[1,0]$. We have an extension of $m'$ to the points $m'(0,0) = m'([0,1],[0,1]) = [0,1]$ and $m'(\infinity , \infinity)=m'([1,0],[1,0]) = [1,0]$.  

Now we want to extend to the points in $Q\intersect l$ which are the image of $(0,\infinity) = ([0,1],[1,0])$ and $(\infinity, 0) = ([1,0],[0,1])$.  Look at the closure of the inverse image of a point $[L_1,L_2]$.  We have $m'^{-1}([L_1,L_2]) =  \{[\xi L_1, y,z, \xi L_2]\ | \xi \neq 0\}$.  Intersecting with the quadric we have $\{L_2 x = L_1 w\}\intersect \{xw-yz = 0\}$, defining a family of $(1,1)$ conics on the quadric surface.  This pencil has two fixed points at $Q \intersect l$.  Near one of them, say $[0,1,0,0]$ we take the coordinates given by $y \neq 0$ have as defining equations$\{L_2 X = L_1W, XW=Z\}$ from which we can see explicitly the proceeding claim, that each inverse image set goes through the origin and is distinguished by its tangent slope.  Further the family of conics degenerates at $[L_1,L_2]=0$ or $\infinity$ to the two lines in the ruling.

Write $\res{R} \rightarrow \proj{1} \times \proj{1}$ for the blowup of $\proj{1} \times \proj{1}$ at the set $\{[0,1,0,0], [0,0,1,0]\}$.  We get a map $\res{m}:B \rightarrow \proj{1}$ by sending each point of the exceptional divisors to the unique point in $\proj{1}$ corresponding to the cubic passing through that point. We have a symmetrical situation at the other point $[0,0,1,0]$ so we can glue the two exceptional divisors of $R_{0,\infinity}$ to each other. Further by gluing the strict transform of the two degeneracy loci $0 \times \proj{1} \cup \proj{1} \times 0 \sim \infinity \times \proj{1} \cup \proj{1} \times \infinity$ we get a lifting of the gluing map $\proj{1} \times \proj{1} \stackrel{\fname{glue}}{\longrightarrow} N \times N$. This gives the commutative diagram
\begin{diagram}
	\res{R} 				&  	\rTo^{\fname{glue}}	& 	R			& \rTo			& \proj{1}		\\
	\dTo 					&						&   \dTo^q  	& \rdTo(2,2)_m	&	\dTo		\\
	\proj{1} \times \proj{1}&   \rTo^{\fname{glue}}	&   N \times N	& \rDashto^m	&	N			
\end{diagram}
giving the extension of the multiplication map.  We can also see that R is just the blowup of the \{singular point of N\}$\times$\{singular point of N\} in the product.
\end{proof}

Now consider a family of elliptic curves over a disk $X \stackrel{\pi}{\rightarrow} D$ degenerating to a nodal cubic over the origin.  We want to extend the multiplication map 
\begin{diagram}
X \times_D X & \rDashto^m & X  
\end{diagram}
In local coordinates the degeneration in the fiber product looks like $\{xy=t=uv\} \subset \C^4$ where t is a coordinate in $D$.  The only singularity of the fiber product is at the point $t=0$ where we have the cone over the Quadric.  There are two choices of small resolution corresponding to a small resolution formed by blowing up a plane in either ruling of the quadric.  One of these choices will correspond to the extension of the multiplication map defined above ( It can be shown that the other corresponds to an extension of the division map).

\subsection{Extension of Multiplication for V}

We proceed in a similar fashion as in the previous case.  The main difficulty in the case for surfaces is that the fiber product of our family with itself is singular along a surface, not at a point but is otherwise analogous.  First we recall some facts about our singular fibers.  

The singular fibers of $\pi$ consist of eight translation scrolls, $T_i$.  A translation scroll is defined by an elliptic curve $E$ embedded as a curve of degree 8 in $\proj{7}$ and a translation element $\rho$.
\begin{eqnarray}\label{DefTranslationScroll}
	T_{E,\rho} = \union_{x\in E} <x,x+\rho > 
\end{eqnarray}
Where $<x,x+\rho >$ is the line in $\proj{7}$ spanned by the point $x$ and $x+\rho$.

A translation scroll is topologically equivalent to $E \times N$ where E is an elliptic curve and N is a nodal cubic. It is singular along the curve E.  In our case, $E \hookrightarrow \proj{7}$ is a curve of degree eight. The smooth locus is a Zariski open set that is a group scheme G (see \cite{HulekKahnWeintraub}).  The group G is given be the short exact sequence
\begin{diagram}
	0 & \rTo & \C^\ast & \rTo & G & \rTo^n & E & \rTo & 0
\end{diagram}

Write $B^0 \hookrightarrow b \isom \proj{1}$ for the locus of $\pi$ with smooth fibers, $V^0 = \pi^{-1}(B^0)$.   On this locus there is a well defined multiplication map $m: V^0 \times_B^0 V^0 \longrightarrow V$ given by $m(x,y) = x+y$.  This gives a rational map
\begin{diagram}
	V \times_B V & \rDashto^m & V
\end{diagram}
We would like to extend this map to all of $V \times_B V$.

\begin{proposition} For a choice of small resolution $\res{V \times_B V} \longrightarrow V \times_B V$ the multiplication map extends to a map 
\begin{diagram}
\res{V \times_B V} & \rTo^{\res{m}} & V
\end{diagram}
\end{proposition}

\begin{proof} 
In what follows we will make use of the following commutative diagram where we define $q$ and $\res{m}$ at an appropriate place.
\begin{diagram}
 		&							&	\res{V_1 \times_B V_2}	&  \rTo^{\res{m}}			& V		\\
		& \ldTo(2,2)^{\res{p}_2}	&				\dTo>q 		&  \rdTo(2,2)^{\res{p}_2}   &		\\ 
	V	& \lTo^{p_1}				&		V_1 \times_B V_2    & \rTo^{p_2}    			& V		\\
		& \rdTo(2,2)_{\pi_1}		&     \dTo>{\pi}            & \ldTo(2,2)>{\pi_2} 		&		\\
		&							&       B                   &                           &		\\
\end{diagram}

Consider the closure of the graph of the multiplication map
\begin{displaymath}
	\Gamma_m = \overline{\{(v_1,v_2,v) \in V\times_B V \times_B V \; | \; m(v_1,v_2) = v\}}
\end{displaymath}
which is a projective variety since it's a subvariety of $V \times_B V \times_B V$. By looking locally at the points where the multiplication map fails we will show that there is a small resolution of the product $\res{V_1 \times V_2} \stackrel{q}{\longrightarrow}V \times_B V $.  Then since $\res{V_1 \times V_2} \subset \Gamma_m$ we will have our projective extension of the multiplication map

We describe the extension of the multiplication map in a discussion running parallel to the case of the nodal cubic.  To each translation scroll $T_i$ with defining elliptic curve $E_{\rho_i}$ we have the projective bundle $S_i$ that is its normalization.  Explicitly, to each point $x$ of $E_{\rho_i}$ we associate the line in $\proj{7}$ passing through $x$ and the point $x+\rho$ for some fixed $\rho$.  $S_i$ is a sub projective bundle of the trivial bundle $\proj{7} \times E_{\rho_i}$.  We will drop the subscript i for clarity.

This bundle has two natural sections over $E_{\rho}$, the zero section $\sigma_0$ given by the intersection with $E$ and the infinity section, $\sigma_\infinity$ given by the other intersection point with the curve $E_{\rho}$. Call the images of these sections $C_0$ and $C_\infinity$ respectively. The map $S \rightarrow T$ is given by gluing the point of $C_\infinity$ over the point $x\in E_\rho$ to the point $x + \rho$.  Viewing $S \subset  \proj{7} \times E_{\rho}$ with projections onto the first and second factors, the gluing map is just projection on to the first factor $S \stackrel{\fname{glue}}{\longrightarrow} \proj{7}$.

Look at the group $G \subset S$ and the maps $G \times G \stackrel{m'}{\longrightarrow} G$. This gives the rational map covering the inclusion of G into T
\begin{diagram}
	S \times S	&	\rDashto^{m'} 	& 	S	\\
	\dTo		&					& \dTo	\\
	T \times T	&	\rDashto^{m}	&  T
\end{diagram}
We can extend this map to $S \times S \backslash \{C_0 \times C_\infinity, \; C_\infinity \times C_0\} $ by taking $m'(\sigma_0(x),s) = \sigma_0(x+n(s))$, where $n:S \rightarrow E$ is the structure map of the projective bundle which corresponds to the group projection $n:G \rightarrow E$.  In a similar way we can define $m'(\sigma_\infinity(x),s) = \sigma_\infinity (x+n(s))$.  

This leaves the sets $C_0 \times C_\infinity$ and $C_\infinity \times C_0$ in $S \times S$ where $m'$ isn't defined. Taking the projection onto the base, we have the inverse image is the translated diagonal, i.e. in the following diagram
\begin{diagram}
	S \times S 	& \rTo^{n \times n} & E \times E	\\
	\dDashto^{m'}	&					&	\dTo^m		\\
	S			&	\rTo^n			&	E
\end{diagram}
The inverse image of a point $x\in E$ is $m^{-1}(x) = \{ (e,f) \in E \times E \; | \; e+f = x\} = \{ (e,x-e) \in E \times E\}$.  Locally on the base then, we can look at the inverse image and see that we have a line crossed with the inverse image from the nodal case.  Taking the blowup along $C_0 \times C_\infinity$ and $C_\infinity \times C_0$ we get get $\res{R} \rightarrow S \times S$ and a morphism $m'':\res{R}\rightarrow S$.  Gluing the two exceptional loci and the strict transform in R of the points $\sigma_0(x) \times S$ with $\sigma_\infinity (x + \rho) \times S$ we get a variety $R$ filling the diagram
\begin{diagram}
	\res{R} 	&  	\rTo^{\fname{glue}}	& 	R			& \rTo			& S		\\
	\dTo 		&						&   \dTo^q  	& \rdTo(2,2)_m	& \dTo	\\
	S \times S	&   \rTo^{\fname{glue}}	&   T \times T	& \rDashto^m	& T			
\end{diagram}

We need to place the above analysis in the context of our fibration.  Looking in the fiber product we can describe the singularity as follows.  Take $t$ to be a coordinate on $R$, $s_1$ a coordinate along $E_{\rho_i}$ the elliptic curve forming the singular locus of the singular fiber.  Near the singular fiber we describe $V_1$ as a subset of $\C^4$ given by $\{(x_1,y_1,s_1,t) \; | \; x_1y_1=t\}$.  Doing similarly for $V_2$ we get our local description of the fiber product near the singularity as
\begin{displaymath}
	\{ ((x_1,y_1,s_1),(x_2,y_2,s_2),t) \subset \C^7 \; | \; x_1y_1=t=x_2y_2 \}
\end{displaymath}
We see that the fiber at $t=0$ has as it's singular locus the cone $x_1y_1=x_2y_2$ in $\C^4$ at every point of the plane $(s_1,s_2)$. Just as in the nodal case there are two choices of small resolution corresponding to the two rulings of the quadric.  The global projective resolution given by the graph $\Gamma_m$ makes a choice of one of these. 

Taking the blowup along a smooth divisor containing the plane of singularities $(s_1,s_2)$ we get the small resolution $\res{V_1 \times_B V_2}$.  The exceptional locus of the small resolution of the cone is a $\proj{1}$, so in this case we get a projective bundle over the plane $(s_1,s_2)$ locally, and globally a $\proj{1}$ bundle on $E_{\rho_i} \times E_{\rho_i}$.
\end{proof}

\section{Some Details On The Singular Fibers $T_i$}

We make some elementary remarks based on our description of T and the extension of the multiplication map.

\begin{proposition}
The pullback of $\OO(H)$ to $S$ is $\OO_{S}(1) \tensor p^\ast \OO_E(4 pt.)$ and the cohomology is given by
\begin{center}
	$\begin{array}{cccc} H^0(S,H) = 8, & H^1(S,H) = H^2(S,H) = 0 \end{array}$ 
\end{center}
\end{proposition}
\begin{proof}
Recall the standard results (see \cite{BeauvilleSurf} Proposition III.18) that

\begin{eqnarray*}
	Pic(S) = p_i^\ast Pic(E_{\rho}) \dsum \Z C_0	&	&	H^2(S,\Z ) = \Z C_0 \dsum \Z F. 
\end{eqnarray*}

Consider $n^\ast H$.  Write $H = aC_0 + bl$, then since $n^\ast H . n^\ast E = n^\ast H . 2C_0 = n^\ast (H.E) = n^\ast (8[pt]) = 8$ we conclude that $H.C_0 = 4$.  Also, $n^\ast H . n^\ast l = n^\ast H . F = n^\ast (H.l) = n^\ast(1[pt]) = 1$.  Hence we get $H = C_0 + 4F$.  
For the cohomology we know that $H^0(S_1,n^*H) = 8$ since it is the polarization given the map $S \rightarrow T \subset \proj{7}$.  Look at the Leray spectral sequence for a fibration with $E^{p,q}_2 = H^{p}(E,R^q\pi_* H) \Rightarrow H^{p+q}(S,H)$.  Since $n^*H.F=1$ we have by the base change and cohomology theorem that for each fiber over a point $y\in E$, $(R^1\pi_* H)_y \isom H^1(\proj{1},\OO_{\proj{1}}(1)) = 0$.  So in particular the spectral sequence degenerates and we get $H^i(S,H) \isom H^i(E,\pi_*H)$.  We can conclude immediately that $H^2(S,H)=0$ since $E$ is a curve.  To calculate that $H^1(S,H)=0$ we consider the Riemann-Roch formula $\euler (H) = deg(ch(H).td(S))_2$.  By standard results (see \cite{BeauvilleSurf} chapter 3) we have that $td(S) = 1 + C_0$, further $ch(H) = 1 + (C_0 + 4F) + 4[pt]$ so $\euler (H) = 8$ and we conclude that $H^1(S,H) = 0$.
\end{proof}

\begin{proposition}\label{multiplication} The n-fold multiplication map $m:T \times ... \times T \rightarrow T$ is defined on all tuples $(y,a_1,a_{n-1})$ where $y$ is any point of $T$ and $a_1,...,a_{n-1}$ are smooth points.  
\end{proposition}
\begin{proof}
This follows from the construction and our first extension of $m$ to the the points in $S \times S$ that are not in the locus $C_0 \times C_\infinity$ and $C_\infinity \times C_0$. 
\end{proof}

\begin{corollary}\label{TPic}For line bundles $\mathcal{L}$ on $T$ and smooth points $a,b \in T$ we have
\begin{eqnarray}
	t^\ast_{a + b}\mathcal{L} \tensor \mathcal{L} = t^\ast_a \mathcal{L} \tensor t^\ast_b \mathcal{L}
\end{eqnarray}
Where $t_a = m(a,\_)$
\end{corollary}
\begin{proof}
This follows the proof given in \cite{HuybrechtsFM} Chapter 9 and is a standard result for Abelian varieties.  In particular on the triple product $T \times T \times T$ for points $(y,a,b)$ with $a$ ,$b$ smooth points we have the three projection maps $\pi_i$ and the pairwise multiplication maps $m_{ij}$.  Take $m:T \times T \times T \rightarrow T$ to be the triple multiplication on points of the form $(y,a,b)$.  Then we have
\begin{displaymath}
	m^* \mathcal{L} \tensor \pi_1^* \mathcal{L} \tensor \pi_2^* \mathcal{L} \tensor \pi_3^* \mathcal{L} = m^*_{12} \mathcal{L} \tensor m^*_{1,3} \mathcal{L} \tensor m^*_{2,3} \mathcal{L}.
\end{displaymath}
Pulling this back via the map $T \rightarrow T \times T \times T$, $y \mapsto (y,a,b)$ we get our result  	
\end{proof}


%

In the following we explicitly see that the space of rank 1 torsion free sheaves on $T_i$ is identified with another translation scroll $\dual{T}_i$.

\section{The Poincare Sheaf on $V \times_B \dual{V}$}

Consider the sheaf on $V \times_B V$ given by
\begin{eqnarray*}
	\KQ = (q_\ast\res{m}^\ast H) \tensor p_1^\ast H^{-1} \tensor p_2^\ast H^{-1}
\end{eqnarray*}
Where $\res{m}$, $q$ are defined in the previous section.  We want the Poincare sheaf on $V \times_B \dual{V}$ to be the pullback of $\KQ$, i.e
\begin{eqnarray*}
	\KQ = (\id \times \phi_H^\ast) \poincare
\end{eqnarray*}
We show that such a sheaf exists by showing that $\KQ$ is $K(H)$ equivariant and descends to the quotient.

\begin{lemma}
	The action of the kernel of $\phi_H$, $K(H)$ on the fiber product $V\times_B V$ via $\id \times K(H)$  extends to an action on the small resolution.  Further it leaves $H$ invariant.
\end{lemma}
\begin{proof}
Elements of $K(H)$ act via translation by elements of order 8 in the elliptic curve, permuting the fibers, and permuting the 64 sections $e_i$.  This action extends to the singular set of $V \times_B V$ just acting as $\id \times \mbox{\{pts. of order 8\}}$.  The automorphism  of the base lifts to an automorphism of the total space of the resolution of $E\times E$, furthermore the action of $K(H)$ is induced by a linear automorphism of $\proj{7}$ and thus fixes $\OO(H)$.
\end{proof}

\begin{proposition} The relative Poincare sheaf, $\poincare$, defined above restricts on the smooth fibers to be the usual Poincare sheaf on $A\times\dA$.  Further, it is trivial when restricted to $\sigma \times \dual{V}$ and $V \times \dual{\sigma}$
\end{proposition}

\begin{proof} 
To define $\poincare$ we need to check that $\KQ = q_\ast\res{m}^\ast H \tensor p_1^\ast H^{-1} \tensor p_2^\ast H^{-1}$ descends to the desired bundle on $V \times_B \dual{V}$.  We follow the steps in \cite{MumfordAbelian} pg 79 and check that the argument extends to our family.  First we check the action of individual elements $k \in K(H)$ lifts to an automorphism, $\psi_a$ of $H$, and then argue that there exists a choice of such lifts such that $\psi_k\psi_{k'} = \psi_{k + k'}$.  On the smooth fibers $K(H)$ acts sheaves via pullback by translation by $k$, $t^\ast_k$ and by \ref{multiplication} we can write the same on the singular fibers.
\begin{eqnarray*}
		t_{(0,k)}^\ast \KQ & = & t_{(0,k)}^\ast ((q_\ast\res{m}^\ast H) \tensor p_1^\ast H^{-1} \tensor p_2^\ast H^{-1}) \\
			& = & (q_\ast \res{t}_{(0,k)}^\ast \res{m}^\ast   H) \tensor t_{(0,k)}^\ast p_1^\ast  H^{-1} \tensor t_{(0,k)}^\ast p_2^\ast H^{-1} \\
			& = & (q_\ast  \res{m}^\ast  t_{k}^\ast H) \tensor p_1^\ast  H^{-1} \tensor p_2^\ast  t_{k}^\ast H^{-1} \\
			& = & (q_\ast\res{m}^\ast H) \tensor p_1^\ast H^{-1} \tensor p_2^\ast H^{-1}
\end{eqnarray*}
Where we used the previous lemma that the action of the kernel fixes H.  So for each element $e_i$ of the kernel there is an automorphism $\psi_{e_i}$ of the line bundle covering the action $V$.  We now need to argue that there is a choice of such automorphisms satisfying $\psi_k\psi_{k'} = \psi_{k + k'}$. Note that every $\psi_{e_i}$ can be changed by an element of $\C^\ast$ and still cover the action of $K(H)$.

There is an isomorphism 
\begin{eqnarray*}
	\KQ\rest_{e \times_B V} & = & (q_{\ast}\res{m}^\ast H)\rest_{e \times_B V} \tensor p_1^\ast  H^{-1}\rest_{e \times_B V}  \tensor p_2^\ast H^{-1}	\rest_{e \times_B V} \\
		& = & H \tensor \OO \tensor H^{-1} \\
		& = & \OO
\end{eqnarray*}
We write the total space of $\OO \rightarrow e \times V$ as  $\C \times V$.  Then we require the lifted automorphism of the action to restrict to the product isomorphism $(c,x) \longmapsto (c, t_k(x))$, where c is in the fiber and $x$ along the base.  Then on the restriction we have $\psi_k \psi_{k'} = \psi_{k + k'}$ and since the $\psi's$ are determined up to scalar multiplication, the $\psi$'s obey this rule on all of $V \times_B V$.

Checking the criteria for the smooth fibers, if $\alpha = \phi(x)$
\begin{eqnarray*}
\poincare\rest_{A \times \dual{\alpha}} & \isom & \phi_H^\ast \poincare \rest_{A \times x} \\
	& \isom &  ((q_\ast\res{m}^\ast H) \tensor p_1^\ast H^{-1} \tensor p_2^\ast H^{-1})\rest_{A \times x} \\
	& \isom &  ((t_{x}^\ast H) \tensor \OO_A \tensor H^{-1}) \\
	& \isom &  \OO_{\alpha}.
\end{eqnarray*}
\end{proof}

\section{The Fourier-Mukai Transform}
For any two varieties  $X\rightarrow B$ and $Y \rightarrow B$ the relative Fourier-Mukai transform is a functor from the derived category of X to the derived category of Y given as
\begin{eqnarray*}
	\FM_\KQ :D(X) \longrightarrow D(Y) \\
	\mathcal{F} \longmapsto Rp_{2\ast} ( p_1^{\ast} \mathcal{F} \tensor \mathcal Q)
\end{eqnarray*}
Where the kernel $\KQ$ is an object in the derived category of $X \times_B Y$, $p_1$ and $p_2$ are the first and second projections respectively.  In the following a sheaf always refers to an object in the derived category concentrated in a single degree.

We have the Poincare bundle $\poincare$ on $V \times_B \dual{V}$ constructed previously with the property that
\begin{eqnarray*}
(id_A \times \phi_H)^\ast\poincare = (q_\ast\res{m}^\ast H) \tensor p_1^\ast H^{-1} \tensor p_2^\ast H^{-1} = \KQ
\end{eqnarray*}

We wish to calculate the Fourier Mukai transform by finding a functor $F:D(V) \to D(\dual{V})$ such that the following diagram commutes
\begin{diagram}
	D(V)           &     \rTo^{\FM_\poincare}  	&        D(\dual{V})    \\
	               &   \rdTo_{\FM_\KQ} 			&        \dTo>F         \\
	               &                  			&          D(V)^{K(H)}
\end{diagram}
Where $D(V)^{K(H)}$ is the derived category with the $K(H)$ equivariant action which we get by tensoring with $\KQ$.

\begin{proposition}The above diagram commutes with $F=\phi_{H\ast}$, the equivariant push forward.  Further $F \circ S_\KQ = S_\poincare$ is an equivalence of categories \end{proposition}
                                                                                                                                                                                            
\begin{proof}
	Clear from the definitions. 
\end{proof}

We will use the above diagram to calculate $S_\poincare$ by calculating $S_\KQ$ and taking the equivariant portion of the bundle.

\section{The Action of the Fourier-Mukai Transform on Cohomology}

We use the fact that the Fourier-Mukai transform descends to a linear map on the chern characters of sheaves in the cohomology rings of $X$ and $Y$, $\fm_\poincare:H^\ast (V,\Q) \longrightarrow H^\ast(\dual{V},\Q)$ .  Since the cohomology of $V$ and $\dual{V}$ have rank 6 (over $\Q$) if we calculate the cohomology of 6 independent sheaves we'll be able to reconstruct $\fm_\poincare$ as a $6 \times 6$ matrix.

\begin{enumerate}
	\item $\OO_A$:\\
		Writing $i:A \hookrightarrow V$ we have that $S_\KQ (i_\ast \OO_A) = \OO_{pt}[-2]$ by standard results for the Fourier-Mukai transform of the Poincar\'{e} sheaf on an abelian surface.  We use Grothendieck-Riemann-Roch to calculate the chern characters
		\begin{eqnarray*}
			ch(i_\ast\OO_A) & = & i_\ast (ch(\OO_A).td(N_{A/V})^{-1}) \\
				& = & i_\ast ( (1).(1) ) \\
				& = & [A]
		\end{eqnarray*}
 		Where we used that $N_{A/V}$ is trivial.

	\item $\OO_e$:\\
		Write $\sigma :B \hookrightarrow V$ where $\sigma(B)=e$, and consider any sheaf $\mathcal{R}$ supported on $B \isom \proj{1}$.  Then $\FM_{\KQ}^i(\sigma_\ast\mathcal{R}) = R^ip_{2\ast}p_1^\ast \sigma_\ast \mathcal{R} \tensor \KQ $.  But $p_1^\ast \sigma_{i\ast}\mathcal{R}$ has support along $\sigma \times_B V_2$ in $V_1 \times_B V_2$, over this locus, $\KQ \isom \OO_{\sigma \times_B V}$, so we have $\FM_\KQ (\sigma_{i\ast} \mathcal{R}) \isom p_{2\ast}p_1^\ast (\sigma_{i\ast} \mathcal{R})$.  By the commutative diagram
		\begin{diagram}
			\sigma_i \times_B V_2   &         \rTo^{p_2}_\isom         &        V_2             \\
	         \dTo<{p_1}             &                                  &        \dTo>{\pi}      \\
	          \sigma_i (B)          &       \lTo^{\isom}_{\sigma_i}    &          B             \\
		\end{diagram}
		We get $ \FM_{\KQ}(\sigma_\ast\mathcal{R}) \isom \pi^\ast \mathcal{R}$.  Take $R=\OO_e$ to get the result.
		
		
		To calculate the map on chern characters again we us Grothendieck-Riemann-Roch
		\begin{eqnarray*}
			ch(\sigma_\ast \OO_e) & = & \sigma_\ast (ch(\OO_e).td(N_{e/V})^{-1}) \\
				& = & \sigma_\ast (1.(1+[pt])) \\
				& = & = [e] + [pt]
		\end{eqnarray*}
		Where we again used that $N_{e/V} = \OO_e(-1) \dsum \OO_e(-1)$		

	\item $\OO_{pt}$:\\
		$S_\KQ ( i_\ast \OO_{pt}) = \OO_A$

	\item $\OO_A(H)$: \\ 
		We have		
		\begin{eqnarray*}
			\FM_\KQ (H)^i & = & R^ip_{2\ast}(p_1^\ast H \tensor \KQ)  \\
				& = & R^ip_{2\ast} ( p_{1\ast} H \tensor m^\ast H \tensor p_1^\ast H^{-1} \tensor p_2^\ast H^{-1} ) \\
				& = & R^ip_{2\ast}(m^\ast H \tensor p_2^\ast H^{-1})  \\
				& = & R^ip_{2\ast}m^\ast H \tensor H^{-1} \\
		\end{eqnarray*}
		So we need to look at $R^ip_{2\ast}m^\ast H.$    To distinguish the first and second factors we write
		\begin{diagram}
			A_1 \times A_2   &   \rTo^m     &    A        \\
			\dTo<{p_2}       &              &             \\
        	  A_2            &              &             \\
		\end{diagram}
		Fix a point $y\in A_2$ and look at the fiber over $y$.  As the fibers of $p_2$ are two dimensional we have that $R^ip_2 m^\ast H = 0$ for $i \geq 3$.  For $i=2$ by the  Base Change and Cohomology theorem we have that the fibers of $R^2p_{2\ast} m^\ast H$ are isomorphic to $H^2(A_{1,y}, m^\ast H \rest_y) \isom H^2(A_{1,y},t^\ast_y H)$.  By Kodaira vanishing we have $H^2(A,t^\ast_y H) = 0$ for all y so $R^2p_{2\ast} m^\ast H=0$ .  In a similar fashion we get $R^1 p_{2\ast} m^\ast H = 0$.  We are left with $R^0 p_{2\ast} m^\ast H$, which we will denote by $p_{2\ast} m^\ast H$.  By the Cohomology and Base Change theorem again we get that that $p_{2\ast} m^\ast H$ is a rank 8 vector bundle on $A$ as $dim(H^0(A_{1,y},t^\ast H))=8$

		For each fiber of this bundle, $H^0(A,t^\ast_y H)$ we have an isomorphism with a chosen fiber $H^0(A_{y,0},H)$ given by translation $t_{y\ast} H^0(A,H) \isom H^0(A_{1,y},t^\ast_y H)$ which is just translation of sections.  Hence $p_{2\ast}m^\ast H$ is trivial and canonically isomorphic to $H^0(A,H)$ and we get a global trivialization.  We conclude $\FM_\KQ ( \OO_A(H) ) = H^0(A,H) \tensor H^{-1}$ 

		We calculate the chern characters of $\OO_A(H)$ and $H^0(A,H) \tensor H^{-1}$.  The chern polynomial of $\OO_A(H)$ is $c_t(\OO_A(H)) = 1 + i^\ast Ht$ from which we calculate the chern character to be $ch(\OO_H(A)) = 1 + i^\ast H + \frac{1}{2} i^\ast H. i^\ast H$.  As $N_{A/V}$ is trivial $td(N_{A/V} = td(N_{A/V})^{-1} = 1$.  So by Grothendieck-Riemann-Roch
		\begin{eqnarray*}
			ch(i_\ast \OO_A(H)) & = & i_\ast (ch(O_A(H).td(N_{A/V})^{-1}) \\
				& = & i_\ast (1 + i^\ast H + \frac{1}{2} i^\ast H. i^\ast H) \\
				& = & [A] + [H.A] + \frac{1}{2} [H.H.A]    \\
  				& = & [A]  + 16[l] + 8[pt]
		\end{eqnarray*}
		The chern polynomial of $H^0(A,H) \tensor H^{-1} \isom \dsum_{j=1}^8 O_A(-H)$ is 
		\begin{eqnarray*}
			c_t & =  & (1-i^\ast H)^8	\\
				& = & 1 - 8[i^\ast H] + 28 [i^\ast H] . [i^\ast H]		
		\end{eqnarray*}
   		Which gives chern character
		\begin{eqnarray*}
			ch(H^0(A,H) \tensor H^{-1}) & = & rk + c_1 + \frac{1}{2}(c_1^2 - 2 c_2) + \dots \\
				& = & 8 - 8 [i^\ast H] + \frac{1}{2}(64 [i^\ast H]^2 - 56 [i^\ast H]^2) \\
				& = & 8 - 8[i^\ast H] + 4 [i^\ast H]^2
		\end{eqnarray*}
		Using Grothendieck-Riemann-Roch, the fact that $N_{A/V}$ is trivial, and the projection formula we have
		\begin{eqnarray*}
			ch(i_\ast H^0(A,H) \tensor H^{-1} ) & = & i_\ast (8 - 8[i^\ast H] + 4 [i^\ast H]^2) \\
				& = & 8[A] - 8 [H.A] + 4 [H.H.A] \\
				& = & 8[A] - 128[l] + 64[pt]
		\end{eqnarray*}
\end{enumerate}
\begin{remark}\label{FMError}
For the last two sheaves we unfortunately have only partial results.  The main difficulty lies in figuring out exactly what $m^*H$ looks like on the singular fibers. For the methods used in this paper these results will not actually matter.  When we construct the matrix, $s_\poincare$, representing the Fourier-Mukai transform will will only need the results from the last three columns representing those classes which have support codimension 2 or greater in V and we've already determined these results in the calculations above.  As we will see, this simplification results from our using a spectral curve, i.e. a line bundle supported on curve in V rather than a more general spectral sheaf with support on surfaces (or even the entire threefold).	
\end{remark}
\begin{enumerate}
	\item[5.] $\OO_V$:\\ 
		We need to calculate $R^ip_{2\ast}\mathcal{F}=R^ip_{2\ast}\KQ$.  First note that $R^ip_{2\ast} \KQ = 0$ for $i \geq 3$ as the fiber dimension is 2.  By \cite{HartshorneAG} Theorem 12.11 $R^2p_{2\ast} \KQ$ has fibers isomorphic to $H^2(V_a,\KQ_a)$.  Choose a point $b \in B$ and look at the fibers over B.  For the case of the smooth fibers we have the fiber of $\KQ$ over $a \in A_1$ is 
		\begin{eqnarray*}
			\KQ \rest_{a \times A_2} & = & (m^\ast H \tensor p_1^\ast H^{-1} p_2^\ast H^{-1})\rest_{a \times A} \\
				& = & (m^\ast H \tensor p_1^\ast H^{-1})\rest_{a \times A} \\
				& \isom & t^\ast_a H \tensor H^{-1} \\
		\end{eqnarray*}
		By a standard result for abelian varieties $H^i(A,\mathcal{F}_a) = 0$ for any non trivial line bundle in $Pic^0(A)$.  Observe that $t^\ast_a H \tensor H^{-1}$ is precisely the map $\phi_H:A \longrightarrow \dA$ and so is trivial for $a \in K(H) \isom \Z_8 \times \Z_8$.  So we see that on the smooth locus of $\pi$, $R^2p_{2\ast} \KQ$ has support on the 64 sections of $\pi$, $e_i$.  By upper-semicontinuity the support is the entire $e_i$, but may contain more components in the singular fibers.

		We consider nonsingular and singular points in $T$ taking first $y \in T$ a nonsingular point.  The the fiber over $y$ is just a copy of $T$ labeled by $y$, $T_y$.  Recall that from the definition of the multiplication map it's defined for all pair $(x,y)$ where $x$ is any point of $T$ and $y$ is a smooth point. Hence we can write by Proposition \ref{multiplication}
		\begin{eqnarray*}
			\KQ\rest_{T_y} & = & t^\ast_y H \tensor \OO \tensor H^{-1}  \\
				& = & t^\ast_y H \tensor H^{-1} 
		\end{eqnarray*}
		By the base change and cohomology theory we have fiber wise isomorphisms $(R^2p_{2\ast} \KQ)\rest_{T_y} \isom H^2(T_y,t^\ast_y H \tensor H^{-1})$. We expect that since $t^\ast_y H \tensor H^{-1}$ is a rank 1 torsion free sheaf on T we have that $H^i(T,t^\ast_y H \tensor H^{-1})=0$ unless $t^\ast_y H \tensor H^{-1}$ is trivial.

		If what we expect is true then on $e_i$ since the fiber dimension is constant we have a line bundle on each $e_i$.  Observe that $\KQ$ restricted to any section is trivial since $(m^\ast H \tensor p^\ast_1 H^{-1} \tensor p^\ast_2 H^{-1})\rest_{V \times e_i} = H \tensor H \tensor \OO = \OO_e$.  So we have over the $e_i$ 
		\begin{eqnarray*}
			\FM_\KQ (\OO_V)\rest_{e_i} & = & R^2 p_{2\ast} \KQ 	\\
				& = & R^2 \pi_\ast \OO_{e_i} 			\\
				& = & (\pi_\ast \omega_\pi)^D 			\\
				& = & \OO_e(-2) 						
		\end{eqnarray*}
		Where in the third step we used relative duality (see \cite{HartshorneRD}).

		We would like to claim that $R^1p_{2\ast}\KQ = 0$ by regularity for the Poincar\'{e} sheaf.  This holds for the smooth fibers but we were  unable to show it for the singular fibers. $R^0p_{2\ast}\KQ = 0$ since otherwise we would have a torsion section of $\KQ$ and there aren't any.
		
		Calculating the chern classes for the our expected Fourier-Mukai we get.
		\begin{eqnarray*}
			ch(\sigma_{i\ast}\OO_B(-2)) & = & \sigma_{i\ast} (ch(\OO_B(-2).td(N_{B/V})^{-1}) \\
				& = & \sigma_{i\ast} ((1-2[pt]).(1+[pt])) \\
				& = & \sigma_{i\ast} (1-[pt])   \\
				& = & [e]-[pt]
		\end{eqnarray*}
		Where we used that since $e$ is the class of an exceptional curve that is the blowup of the cone in $\C^4$ it has normal bundle $N_{e/V} = \OO_e(-1) \dsum \OO_e(-1)$

	\item[6.] $\OO_V(H)$:\\
		Calculating
		\begin{eqnarray*}
			\FM_\KQ^i ( \OO_V(H)) & = & R^ip_{2\ast} (p_1^\ast H \tensor q_\ast m^\ast H \tensor p_1^\ast H^{-1} \tensor p_2^\ast H^{-1} \\
			& = & R^ip_{2\ast} ( q_\ast m^\ast H \tensor p_{2\ast} H^{-1})             \\
			& = & R^i(p_{2\ast} q_{\ast} \res{m}^\ast H) \tensor H^{-1}
		\end{eqnarray*}
		So we need to calculate $(p_{2\ast} q_\ast \res{m}^\ast H)$.  From the calculation of $\OO_A(H)$ we know that restricted to the smooth fibers we have this $(p_{2\ast} q_\ast \res{m}^\ast H)=H^0(A,H)$ which we can think of as the sheaf $\pi^\ast \pi_{\ast} H$ restricted to $A$.  We expect therefore that $(p_{2\ast} q_\ast \res{m}^\ast H) \isom \pi^\ast \pi_{\ast} H$


		The total chern class for $\OO_V(H)$ is $c_t(\OO_V(H)) = 1 + [H]$ so the chern characters are given by 
		\begin{eqnarray*}
			ch(\OO_V(H)) & = & 1 + [H] + \frac{1}{2}[H]^2 + \frac{1}{6} [H]^3 \\
				& = & 1 + [H] + 8[e] + 8[l] + \frac{8}{3}[pt]
		\end{eqnarray*}
		We calculate the chern character.
		\begin{eqnarray*}
			ch(\pi^\ast \pi_\ast \OO_A(H)\tensor O_V(-H)) & = & rk + c_1 + \frac{1}{2}(c_1^2 - 2c_2) + \frac{1}{6}(c_1^3 - 3c_1c_2 + 3c_3) \\
				& = & 8 - 8[H] + 4[H]^2 - \frac{4}{3}[pt] \\
				& = & 8 - 8[H] + 64[e] + 64[l] - \frac{64}{3}[pt].
		\end{eqnarray*}

\end{enumerate}

We summarize the results in the following tables\footnote{The last two rows of both tables are our expected results.  See Remark \ref{FMError}}\newpage
\begin{table}[ht]\caption{Summary Of Fourier-Mukai On A Basis Of Sheaves}
\begin{center}
\resizebox{\linewidth}{!}{%
\begin{tabular}{|c|c|c|c|} \hline
	$\displaystyle\mathcal{F}$ 	& $ch(\mathcal{F})$ 	&$\FM_\KQ(\mathcal{F})$	&	$ch(\FM_\KQ(\mathcal{F}))$		\\ \hline \hline
	 $\OO_A$    				& $[A]$ 				& $\dsum_{i=1}^{64}\OO_{pt}$  &$64\OO_{pt}$     			\\ \hline
	 $\OO_e$    				& $[e] + [pt]$ 			& $\OO_V$     &$[V]$                   			\\ \hline
	 $\OO_{pt}$ 				& $[pt]$ 				& $\OO_A$     &$[A]$              					\\ \hline
	 $\OO_A(H)$ 				& $[A] + 16[l] + 8[pt]$ & $H^0(A,H)\tensor H^{-1}$ 		& $8[A]-128[l]+64[pt]$ 		    \\ \hline
	 $\OO_V$    				& $[V]$ 				& $\dsum_{i=1}^{64} \sigma_{i\ast}\OO_{B}$ & $64[e]-64[pt]$	\\ \hline
	 $\OO_V(H)$ 				& $[V]+[H]+8[e]+8[l]+\frac{8}{3}[pt]$ & $(\pi^\ast\pi_\ast\OO_V(H))\tensor H^{-1}$ & $8[V]-8[H]+64[e]+64[l]-\frac{64}{3}[pt]$ \\ \hline
\end{tabular}
}
\vspace{10mm}

\begin{tabular}{|c|c|c|} \hline
	 $\displaystyle\mathcal{F}$ 	& $\FM_\poincare (\mathcal{F})$ & $ch(\FM_\poincare (\mathcal{F}))$ \\ \hline \hline
	 $\OO_A$    				& $\OO_{pt}$    					& $[pt]$       						\\ \hline
	 $\OO_e$    				& $\OO_{\dual{V}}$     			& $[\dual{V}]$       				\\ \hline
	 $\OO_{pt}$ 				&$\OO_{\dA}$      				& $[\dA]$    						\\ \hline
	 $\OO_A(H)$ 				&							& $8[\dA]-16[\dE]+[pt]$ 			\\ \hline
	 $\OO_V$    				& $\dual{\sigma}_\ast \OO_{B}$	&  $[\de]-[pt]$ 					\\ \hline
	 $\OO_V(H)$ 				&    			& $8[\dual{V}] -[\dH] +[\de] + 8[\dE] -\frac{1}{3}[pt]$ \\ \hline
\end{tabular}

\end{center}
\end{table}
\vspace{10mm}

From which we can calculate the map $s:H^\ast(V,\Q) \longrightarrow H^\ast(\dual{V},\Q)$ on the chern characters.  In the basis of $V$ and $\dual{V}$ given by $\{[V],[H],[A],[e],[l],[pt] \}$ and \newline $\{[\dual{V}],[\dH],[\dA],[\de],[\dE],[pt]\}$ respectively we have\footnote{The first two columns of $\fm_\poincare$ are the expected result. In what follows only the last four columns are relevant to our calculation. See Remark \ref{FMError}} 

\begin{proposition} The map induced on cohomology $s_\poincare :H^\ast(V,\Q) \longrightarrow H^\ast(\dual{V},\Q)$ and $s^{-1}_\poincare  :H^\ast(\dual{V},\Q) \longrightarrow H^\ast(V,\Q)$ are given by the matrices
\begin{displaymath}
	s_\poincare = \left( \begin{array}{cccccc}
			0	&	0	&	0	&	1	&	0	&	0 	\\
			0	&	-1	&	0	&	0	&	0	&	0	\\
			0	&	16/3	&	0	&-1	&	0	&	1	\\
			1	&	0	&	0	&	0	&	0	&	0	\\
			0	&	16	&	0	&	0	&	-1	&	0	\\
			-1	&	2/3	&	1	&	0	&	0	&	0	
		\end {array} \right)
\;\;\;\;
	s_\poincare^{-1}= \left( \begin {array}{cccccc}
			0	&	0	&	0	&	1	&	0	&	0	\\
			0	&	-1	&	0	&	0	&	0	&	0	\\
			0	&	2/3	&	0	&	1	&	0	&	1	\\
			1	&	0	&	0	&	0	&	0	&	0	\\
			0	&	-16	&	0	&	0	&  -1	&	0	\\
			1	&  16/3	&	1	&	0	&	0	&	0	
		\end {array} \right)
 \end{displaymath}
\end{proposition}

\section{Stability of Bundles and Spectral Curves}

Although we will not be able to completely classify which bundles are stable we can offer some numerical conditions on when stability is guaranteed.  The proof is essentially the same as for elliptic surfaces and we follow \cite{FriedmanMorganWitten}.  First we establish some definitions and basic correspondences between certain sub-bundles of $\mc{E}$ and their spectral covers.

Consider a spectral curve $C \hookrightarrow V$ with a line bundle $\mc{L}$ supported on $C$.  Associated to this data via the Fourier-Mukai transform is a vector bundle $\mc{E}$ or rank $r$ on the dual fibration.  Write $C_{r'}$ for the spectral curve associate with $\wedge^{r'} \mc{E}$.

\begin{definition}\label{AbsIrr}We will say that the spectral data {\it $r'-$irreducible (resp. $r'-$reducible)} if $C_{r'}$ is irreducible (resp. reducible) and is {\it absolutely irreducible (resp. reducible)} if $C_{r'}$ is irreducible (reducible) for all $0 < r' <r$.
\end{definition}

\begin{proposition}\label{Rank1Reducible} Let $\mathcal{F}$ be a rank 1 torsion free sheaf on $\dV$ with $c_1(\mc{F}).A = 0$ and $\mathcal{E}$ a rank $r > 1$ bundle whose restriction to a generic fiber is flat semi-stable.  Then if there exists a map $0 \rightarrow \mc{F} \rightarrow \mc{E}$ the spectral cover associated to $E$ is reducible.
\end{proposition}
\begin{proof}
We can write $\mc{F} = I_Z \tensor O_\alpha$ where $Z \subset \dV$ is a sub-scheme of codimension $\geq 2$ and $\OO_\alpha$ is in $\Pic^0(\dV)$.

Restricted to a generic fiber A the ideal sheave $I_Z$ has codimension two, since if were a curve generically then Z would be a surface.  Hence we can take $I_Z\rest_A$ to be a collection of points 0-dimensional sub-schemes $\{z_1,...z_k\}$.  On A consider the short exact sequence
\begin{equation*}
	\ses{\OO_\gamma \tensor I_Z}{\OO_\gamma}{\dsum \C_{z_i}}
\end{equation*}
Taking the Fourier-Mukai transform we get the long exact sequence
\begin{diagram}
	0 & \rTo & S^0(\OO_\gamma \tensor I_Z) & \rTo & S^0(\OO_\gamma) & \rTo & S^0(\dsum \C_{z_i}) & \rTo &	\\
	  & \rTo & S^1(\OO_\gamma \tensor I_Z) & \rTo & S^1(\OO_\gamma) & \rTo & S^1(\dsum \C_{z_i}) & \rTo &	\\
	  & \rTo & S^2(\OO_\gamma \tensor I_Z) & \rTo & S^2(\OO_\gamma) & \rTo & S^2(\dsum \C_{z_i}) & \rTo & 0	\\	
\end{diagram}
But knowing the Fourier-Mukai transform of the first two terms we can write
\begin{diagram}
	0 & \rTo & S^0(\OO_\gamma \tensor I_Z) & \rTo & 0    	  & \rTo & \dsum \OO_{\alpha_i} & \rTo &	\\
	  & \rTo & S^1(\OO_\gamma \tensor I_Z) & \rTo & 0    	  & \rTo & 0			& \rTo &	\\
	  & \rTo & S^2(\OO_\gamma \tensor I_Z) & \rTo & \C_\gamma & \rTo & 0			& \rTo & 0	\\	
\end{diagram}
From which we conclude that $S^0(\OO_\gamma \tensor I_Z)=0$, $S^1(\OO_\gamma \tensor I_Z) \isom \dsum \OO_{\alpha_i}$, and $S^2(\OO_\gamma \tensor I_Z) \isom \C_\gamma$.

Now consider the case where we have the short exact sequence
\begin{equation*}
	\ses{\OO_\gamma \tensor I_Z}{\mc{E}}{\mc{Q}}
\end{equation*}
On a generic fiber we have $\mc{E}\rest_A \isom \OO_{\alpha_1}\dsum ... \dsum \OO_{\alpha_r}$.  Since we have a nonzero map $\OO_\gamma \tensor I_Z \rightarrow \OO_{\alpha_1}\dsum ... \dsum \OO_{\alpha_r}$ we know that $\gamma = \alpha_i$ for at least some i.  Moreover by suitable generic choices we can assume that $\OO_{\alpha_i} \neq \OO_{\alpha_j}$ for $i \neq j$.

From the above short exact sequence we get the long exact sequence whose last terms are
\begin{diagram}
	0  & \rTo & \C_{\alpha_i} & \rTo & \C_{\alpha_1} \dsum ... \dsum \C_{\alpha_r} & \rTo & S^2(Q) & \rTo & 0 
\end{diagram}
But this is happens over a Zariski open set of the base and hence the spectral curve associate to $\mc{E}$ has a component associated to $\OO_\gamma \tensor I_Z$.
\end{proof}

\begin{proposition} If there is a sheaf $\mc{F}$ such that $\mc{F} \rightarrow \mc{E}$ injective and on a generic fiber $A$, $\slope_D(\mc{F}) = 0$ (i.e. $c_1(\mc{F}).A.D=0$), then $\mc{E}$ is $r'$-reducible, where $r'=\rk{\mc{F}}$.
\end{proposition} 
\begin{proof}
The case where $\rk{\mc{F}} = 1$ follows from proposition \ref{Rank1Reducible} above.  For the case $\rk{\mc{F}} = r' < r$ reduce to the rank 1 case by considering $\wedge^{r'}\mc{F}$ 
\end{proof}

\begin{theorem} Let $\mc{E}$ be given as the Fourier-Mukai transform of a line bundle supported on a curve such that $\mc{E}$ is absolutely irreducible (see definition \ref{AbsIrr}, and let $H_0$ be a fixed ample Divisor. Write $D_k = H_0 + kA$.  Then for $k>>0$ we have that $\mc{E}$ is stable with respect $D$.
\end{theorem}

\begin{proof}
Let $\mathcal{F}$ be a subsheaf of $\mathcal{E}$ with the $0 < \mbox{rank}(\mathcal{F}) < \mathcal{E}$.  Then since $\mc{F}$ generically injective we have $\mc{F}\rest_A$ injective into $\mc{E}\rest_A$ and since the later is semistable we have that.
\begin{equation*}
	\slope_D(\mc{F}\rest_A) \leq \slope_D(\mc{E}\rest_A) = 0
\end{equation*}
From which we get that $D.c_1(\mc{F}).A \leq 0$ and  $c_1(\mc{F})$ is not an effective divisor. By absolute irreducibility we have a strict inequality and we note that $\slope_D(\mc{F}\rest{A}) = \frac{D.c_1(\mc{F}).A}{\rk{\mc{F}}}$ is strictly negative number bounded above by $\frac{-1}{n-1}$.  We'll need the following lemma:

\begin{lemma}
	There is a constant $a$, depending only on $\mc{E}$, $H_0$ and $A$ such that
	\begin{equation*}
		\frac{c_1(\mc{F}).H_0^i.A^{2-i}}{\mbox{rank}(\mc{F})} \leq a
	\end{equation*}
\end{lemma}
for all subsheaves $\mc{F} \subset \mc{E}$ and $i=0,1,2$.

\begin{proof}
	For $i = 0$ the above expression is bounded by $0$ since $A^2=0$.

	For $i=0,1$ consider the following. There is a filtration of $\mc{E}$
	\begin{equation*}
		0 \subset \mc{E}_0 \subset ... \subset \mc{E}_{n-1}=\mc{E}
	\end{equation*}
	such that successive quotients are torsion free rank 1.  Hence the successive quotients are of the form $L_i \tensor I_{Z_i}$ here $L$ is a line bundle and $I_Z$ is an ideal sheaf of a codimension 2 (or greater) subscheme of V.

	First we consider the case that rank$(\mc{F}) =1$.  Since $0 \rightarrow \mc{F} \rightarrow \mc{E}$ we have that for some $i$ there is a nonzero map $\mc{F} \rightarrow L_i \tensor I_Z$, hence we can write $\mc{F}$  as $L_i \tensor \OO(-K) \tensor I_X$ for some effective divisor $K$ and subscheme X of codimension 2 or greater.  Then
	\begin{equation*}
		c_1(\mc{F}).H_0^i.D^{2-i} \leq c_1(L_j).H_0^i.D^{2-i}
	\end{equation*}
 	and we get a bound independent of ${\mc{F}}$.

	Consider now the case that rank$(\mc{F})=r$.  Then $det(\mc{F}) = \wedge^r \mc{F}$ and we have a nonzero map $\wedge^r \mc{F} \rightarrow \wedge^r \mc{E}$.  We proceed as in the previous case by consider a filtration of $\wedge^r \mc{E}$ such that the quotients are again rank 1 sheaves.
\end{proof}

Continuing with the main theorem with have that by the previous lemma
\begin{eqnarray*}
	\slope_{D}(\mc{F}) & = & \frac{ c_1(\mc{F}).D^2 }{ \rk{\mc{F}} } 		\\
			& = & \frac{ c_1(\mc{F}).(H_0 + kA).(H_0 + kA) }{ \rk{\mc{F}} }	\\
			& \leq & a - \frac{2k}{n-1}						
\end{eqnarray*}
Recalling that $c_1 (V)$ is supported just on the fibers $A$ we have that
\begin{eqnarray*}
	\slope_{D}(\mc{E}) & = & \frac{ c_1(\mc{E}).D^2 }{ \rk{\mc{E}} } 	\\
			& = & \frac{ c_1(\mc{E}).(H_0^2 + k H.A) }{n} 	\\
			& = & \frac{ c_1(\mc{E}).H_0^2 }{n}		\\
			& \leq & a 
\end{eqnarray*}
Hence for some sufficiently large k we have 
\begin{equation*}
	\slope_D(\mc{F}) < \slope_D(\mc{E})
\end{equation*}
for all choices of subsheaves $\mc{F}$.  Hence $\mc{E}$ is slope stable.
\end{proof}

\section{Future Directions}\label{FutureDirections}

%
%

There are several related families of abelian surface fibered Calabi-Yaus to which the techniques and results developed here can be applied providing a rich bestiary of examples.  One can look at intermediate quotients by taking quotients of subgroups of $\Z_8 \times \Z_8$.  In \cite{GrossPopescu} they describe another family of $(1,8)$ polarized abelian surfaces related to ours by flopping the small resolution (called $V^2_{8,y}$ in that paper).  This flop has second abelian surface fibration whose fibers carry either $(2,8)$ or a $(4,4)$ polarization.  Further In \cite{BorisovHua} they describe non-commutative groups of order 64 acting on $V$ giving quotients with non-abelian fundamental groups.  One can attempt to build bundles on these as well as intermediate quotients.

There are additional families of abelian surface fibrations arising in a manner similar to how $V_{8,y}$ is defined \cite{GrossPopescu}.  In particular there are  $(1,6)$, $(1,7)$ and $(1,5)$ polarized abelian surface fibrations with Heisenberg groups acting to give interesting quotient varieties.  One could hope that bundles could be built on these directly using the spectral construction without the need to take elementary transformations (which cause the main difficulty in the cases we examine here). 

\section{Some Facts Regarding Abelian Surfaces}
In this appendix we include some small results that are used in the main part of the discussion.  First we prove some well known result for abelian surfaces.  In the second section we give some simple results for sheaves on algebraic varieties.  

\subsection{Some Facts Regarding Abelian Surfaces}
We collect in this Appendix some facts on abelian surfaces.  These results were primarily used in the analysis that went into the companion paper \cite{BakBouchardDonagi}.  All of these facts are probably well known. Our primary references are \cite{MumfordAbelian} and \cite{BirkenhackeLange}.

 


\begin{proposition}\label{SL2ZTRANS}  The $SL(2,\Z)$ action on $A = E \times E$ given by
\begin{eqnarray*}
	E \times E & \longrightarrow & E\times E \\
	(x,y) & \longmapsto & (ax+cy,bx+dy)
\end{eqnarray*}
acts transitively on the generators of the effective cone of A.
\end{proposition}
\begin{proof}
For $a,b$ relatively prime we have the slope map $m_{a,b}: E \rightarrow E \times E$ taking $x \mapsto (ax,bx)$.  Write $E_{a,b}$ for the image of this map.  We want to describe $E_{a,b}$ as a linear combination of $E$,$F$, and $\Delta$.  Write $e_i$, $i=1,...,4$ a basis of the lattice $\Z^4 \hookrightarrow \C^2$.  Where $e_1, e_2$ span the lattice of the first factor of $A=E\times E$ and $e_3, e_4$ the second factor.  In this basis we can represent E, F and $\Delta$ as
\begin{eqnarray*}
	E=e_1 \wedge e_2 & F = e_3 \wedge e_4 & \Delta = (e_1 + e_3) \wedge (e_2 + e_4)
\end{eqnarray*}
Now, $E_{a,b} = \{ (ax,bx) \}$ and lifting to $\C^4$ we can describe it as the plane spanned by 
\begin{eqnarray*}
	\lambda_1 = \left( \begin{array}{c}  a \\ 0 \\ b \\ 0  \end{array} \right) &
	\lambda_2 = \left( \begin{array}{c}  0 \\ a \\ 0 \\ b  \end{array} \right)
\end{eqnarray*}
where $\lambda_1$ and $\lambda_2$ are in the lattice.  So $E_{a,b}$ is represented by
\begin{eqnarray*}
	\lambda_1 \wedge \lambda_2 & = & (a e_1 + b e_3) \wedge (a e_2 + b e_3)   \\
		& = & a^2 e_1 \wedge e_2 + ab(e_1 \wedge e_4 + e_3 \wedge e_2) + b^2 e_3 \wedge e_4
\end{eqnarray*}
So if $E_{a,b} = xE + yF + x\Delta$ we get
\begin{center}
\begin{tabular}{lllll}
	$E_{a,b} E $		& $=$ & 	$b^2$	 	& $=$ & $y + z$\\
	$E_{a,b}.F$ 		& $=$ & 	$a^2$ 	& $=$ & $x + z$\\
	$E_{a,b}.\Delta$ 	& $=$ & 	$(a-b)^2$ & $=$ & $x + y$.
\end{tabular}
\end{center}
Solving this linear system we find $E_{a,b} = a(a-b) E + b(b-a) F + ab \Delta$.  But $E_{a,b}$ is just the image of $E$, the first factor in $E\times E$ under our $SL(2,\Z )$ action.  Making a similar calculation for $F$ and $\Delta$ we find that the action can be written as
\begin{displaymath}
	\left( \begin{array}{ccc}
	a(a-b)	& 	c(c-d)	&	(a+c)(a+c-b-d)	\\
	b(b-a)	&	d(d-c)	&	(b+d)(b+d-a-c)	\\
	ab		&	cd		&	(a+c)(b+d)		\\ 
	\end{array} \right)
\end{displaymath}

Now we want to write down the generators of the effective cone and check that the action is transitive.  Recall that the effective cone generators is the shortest vectors vectors $(l,m,n) \in \Z^2$ such that $lm + ln + mn = 0$  These are in 1-1 correspondence with the rational points of the conic $\{xy+xz+yz=0\}\subset \proj{2}$.  Looking at the $z \neq 0$ patch we have $\{(X,Y) | XY +X+Y = 0\} = C$.  Consider the line through $(0,0) \in C$ with rational slope, $Y = tX$.  It intersects C at one other point given by substituting $tX$ for $Y$ in the above equation.  We find $X=0$ or $X=-(1+ \frac{1}{t})$   Using the equation for C we $Y=t+1$.  This gives a parameterization of the rational points in $\proj{2}$ as
\begin{displaymath}
	\{[(-(1+\frac{1}{t}):-(t+1)):1] \; | \; t \in \Q\backslash 0 \} \union \{[1:0:0],[0:1:0]\} 
\end{displaymath}
Where $[1:0:0]$ and $[0:1:0]$ are the images of $E$ and $F$ respectively.

The orbit of $E$ is given by $E_{a,b} = a(a-b) E + b(b-a)F + ab \Delta$.  Assuming $ab \neq 0$ this corresponds to a point of the form
$[a(a-b):b(b-a):ab] = [\frac{a}{b} - 1: \frac{b}{a} -1: 1]$ and taking $\frac{a}{b} = -\frac{1}{t}$ we get precisely the rational points of C.  Furthermore taking $b=1, a=0$ we find $E_{0,1} = F$.  Hence the entire effective cone is in the orbit of E.
\end{proof}

\begin{lemma}\label{MinIntersection}Let A be an abelian surface. Then if $D$ and $H$ are effective divisors such that $D.D = d$ and $H.H = h$ then
\begin{eqnarray}
	D.H \geq \sqrt{dh}
\end{eqnarray}
\begin{proof}
	First recall that for an abelian surface if D is effective then $D^2 \geq 0$.  By the Hodge Index theorem the intersection form is lorentzian, i.e. there is a choice of basis so that the intersection form is of the form $diag(1,-1,...,-1)$. But for vectors vectors of positive norm in a lorentzian metric there is a reverse Schwarz inequality, namely 
	\begin{displaymath}
		(H.D)^2 \geq (H.H)(D.D).
	\end{displaymath}
	The result follows immediatly.

%
\end{proof}

\end{lemma}

\begin{lemma} For a flat bundle $S$ and line bundle $Q$ on an abelian surface A.  Then a necessary condition for $Hom(S,Q) \neq 0$ is that $H^0(Q \tensor \OO_{\alpha}) \neq 0$ some $\OO_{\alpha}$. 
\end{lemma}
\begin{proof}
 Write $S \isom \Dsum_j S_j^{k(j)}$ where $S_j^{k(j)}$ is a $k(j)$ level extension of $\OO_A$, i.e. there is a sequence of extensions (up to tensoring with some flat line bundle)
\begin{eqnarray*}
	& \ses{ \OO }{ S_j^{k} }{ S_j^{k-1} }	&	\\
 	& \ses{ \OO }{ S_j^{k-1} }{ S_j^{k-2} }	&	\\
	& \vdots								&	\\
	& \ses{\OO}{S_j^1}{\OO}					&
\end{eqnarray*}
Now the obvious induction argument will show that $Hom(S,Q) \neq 0$ if and only if $Hom(\OO,Q) = \neq 0$ (again up to tensoring with some $\OO_{\alpha}$.
\end{proof}

\end{document}